\documentclass[12pt]{article}

\usepackage[utf8]{inputenc}

\usepackage{hyperref}
\usepackage{optidef}
\usepackage{amsmath,amssymb,amsfonts,amsthm,enumerate}
\usepackage{eucal} 
\usepackage{xcolor}
\usepackage{youngtab}
\usepackage{young}
\usepackage{url}
\usepackage{lscape}
\usepackage{tikz}
\usepackage{environ}
\usepackage{caption}
\usepackage{subcaption}
\usepackage{algorithm}
\usetikzlibrary{positioning, backgrounds}
\usepackage[noend]{algpseudocode}
\usepackage{graphicx} % Required for inserting images
\usepackage{amsmath,amsthm,amssymb,amscd}
\usepackage[export]{adjustbox}

\DeclareMathOperator{\proj}{proj}

\DeclareMathOperator{\Alpha}{Alpha}

\newcommand{\timesteps}{\mathcal{T}}
\newcommand{\neurons}{\mathcal{N}}

\newcommand{\betti}{\beta}

\newcommand{\grad}{\nabla}

\newcommand{\dataset}{\mathcal{D}}

\newcommand{\thesphere}{S^{d-1}}

\newcommand{\domx}{\mathcal{X}}
\newcommand{\domy}{\mathcal{Y}}

\DeclareMathOperator{\supp}{supp}

\DeclareMathOperator{\im}{Im}
\DeclareMathOperator{\chull}{conv}

\newcommand{\ret}{F}
\newcommand{\hemi}{H}
\newcommand{\hemistrict}{H^{+}}
\newcommand{\rplus}{\mathbb{R} \cup \{\infty\}}

\newtheorem*{thma}{Theorem A}

\newtheorem{thm}{Theorem}

\newtheorem{lemma}{Lemma}

\newtheorem{prop}{Proposition}
\theoremstyle{definition}
\newtheorem{defn}{Definition}
\newtheorem{remark}{Remark}

\theoremstyle{definition}

\newtheorem{alg}{Algorithm}

\counterwithin{figure}{subsection}

\counterwithin{equation}{section}

\counterwithin{alg}{section}

\title{Alpha shapes and optimal transport on the sphere}

\author{Erik Carlsson and Greg DePaul}
\begin{document}
\maketitle

\begin{abstract}
In \cite{carlsson2024kernel}, the authors used 
the Legendre transform to give a tractable method for studying Topological Data Analysis (TDA) in terms of sums of Gaussian kernels.
In this paper, we prove a variant
for sums of cosine similarity-based kernel functions,
which requires considering 
the more general ``$c$-transform'' from optimal transport theory \cite{villani2008old}. 
We then apply these methods to a point cloud 
arising from a recent breakthrough study, which exhibits a toroidal structure in the brain activity of rats \cite{Gardner2022}. A key part of this application 
is that the transport map and 
transformed density function arising from the theorem replace certain delicate preprocessing steps related to density-based denoising and subsampling.

\end{abstract}

\section{Introduction}

Let $\dataset=\{x_i\}_{i=1}^N\subset \mathbb{R}^d$ be a point cloud, and let $h>0$ be a positive number called a scale,
or bandwidth parameter. Then a Gaussian kernel density estimator is a function $f:\mathbb{R}^d\rightarrow \mathbb{R}_{+}$ of the form
\begin{equation}
\label{eq:gausskde}    
f(x)=\sum_{i} a_i \exp(-\lVert x-x_i\rVert^2/2h^2)
\end{equation}
for some coefficients $a_i>0$. The function $K_h(x,y)=\exp(-\lVert x-y\rVert/2h^2)$ is called the kernel function. If they can 
be computed, the topological types of the super-level sets $\{f\geq c\}$ can be taken as 
a robust definition of the ``shape'' 
of $\dataset$ \cite{bubenik2012statistical,phillips2015geometric,chazal2011clustering}.
The problem is that it is impractical to model them by a (filtered) simplicial complex in all but very low dimension, because the number of vertices required is related to the $\epsilon$-covering number, which scales exponentially with the embedding dimension $d$.
In other words, the kernel functions introduce a problem with the \emph{curse of dimensionality}, even in the case of a single data point.

In \cite{carlsson2024kernel}, J. Carlsson and the first author presented a method to resolve this problem using the Legendre transform. 
Setting $E(x)=-\log(f(x))$, it was shown that there exists a transformed function $\alpha$ 
on the interior of the convex hull $\chull(\dataset)$, which
satisfies
\begin{equation}
\label{eq:kdetilde}    
E(x)=\inf_{y\in \mathbb{R}^d} \alpha(y) +\lVert x-y\rVert^2/2h^2.
\end{equation}
It was proved that the inclusion of the sublevel sets
$\{\alpha \leq a\}\subset \{E\leq a\}$ induces a homotopy equivalence with an explicit inverse, so that we may consider the smaller one without affecting the topology. On the other hand, since 
$\{\alpha\leq a\} \subset \chull(\dataset)$ 
in an explicit way, we see that 
their covering numbers are unaffected by taking linear isometries $\mathbb{R}^d \hookrightarrow \mathbb{R}^{d'}$, and so do not suffer from the above scalability issues.

The natural simplicial complex on 
which to model $\{\alpha\leq a\}$
is a fundamental construction from computational topology 
known as an \emph{alpha complex} \cite{edelsbrunner1983shape,edelsbrunner1995union,edelsbrunner1992weighted,edelsbrunner2010computational}, whose geometric realizations are called \emph{alpha shapes}. Using a recent algorithm
based on the powerful
duality principle
in mathematical optimization
\cite{carlsson2023alpha}, it becomes possible to compute these complexes in higher dimension, resulting in refined geometric models of density landscapes.
Interestingly, the map that induces the homotopy 
inverse is important for selecting the 
vertices of the alpha complex.

While we have framed this problem in terms of persistent homology, there are broader problems related to robust subsampling and denoising, or to determining a geometric cover of a point cloud with varying density.
%When persistent homology itself is not the focus, a statement relating to 
%homotopy equivalence can still be taken as a rigorous certificate that such a selection scheme captures the overall shape of $\dataset$. 

%These results provide a robust measure of how well the %selected landmarks capture the underlying geometric or %topological structure of the data.

\subsection{Spherical kernel functions}

In this paper, we study an extension of the results of \cite{carlsson2024kernel} to spherical point clouds
$\dataset\subset S^{d-1}$, and
a kernel function based on cosine similarity:
%sums on the sphere.
\begin{equation}
\label{eq:introspherekde}    
f(x)=\sum_{i} a_i \max(x\cdot x_i,0)^t,
\end{equation}
the multiplication being the usual dot product.
For spherical data, Gaussian kernels are not well-suited because there is an undesirable lower bound of $K_h(x,y) \geq e^{-2/h^2}$ for all points on the sphere, since the distance between antipodal points is 2. 
%If the scale parameter is fairly large, such as $h=.4$, it can happen that the point of highest density is the origin. 
By contrast, the
kernel function \eqref{eq:introspherekde} is zero on all pairs of points which are orthogonal or form an obtuse angle.
%One source of spherical point clouds comes from taking 
%$\{(\sqrt{\pi_1},...,\sqrt{\pi_n})\}\subset S^{n-1}$
%where $\pi$ ranges over a collection of finite probability distributions. In this case, the kernel in \eqref{eq:introspherekde} is a power of the Bhattacharyya coefficient, which is zero for distributions with disjoint support \cite{bhattacharyya1946measure}.

In order to extend the density results to the sphere, we require a concept known as $c$-convexity, as defined in Villani's well-known book \cite{villani2008old}: let $\domx$ be a space, and let
$c:\domx\times \domy\rightarrow \mathbb{R}\cup\{\infty\}$ be a function called a \emph{cost function}. A function
$\psi:\domx\rightarrow \mathbb{R}$ is called \emph{$c$-convex} if there exists a conjugate function $\psi^c:\domy\rightarrow \mathbb{R} \cup \{-\infty\}$ satisfying
\begin{equation}
\label{eq:introcconv}    
\psi(x)=\sup_{y\in \domy} (\psi^c(y)-c(x,y)),\ \ 
\psi^c(y)=\inf_{x\in \domx} (\psi(x)+c(x,y)).
\end{equation}
In terms of \eqref{eq:kdetilde}, 
$\psi(x)=-E(x)$ would be $c$-convex
for $c(x,y)=\lVert x-y\rVert^2/2h^2$
with $c$-transform $\psi^c(y)=-\alpha(y)$.
In optimal transport, $c$-convex functions can be used to construct transport maps
$T: \domx \rightarrow \domy$, according to the rule that $y=T(x)$ when $y$ attains the supremum in \eqref{eq:introcconv}.

We can now state our main result.
\begin{thma}
\label{thm:intro}
Let $f(x)$ be as in \eqref{eq:introspherekde} for a fixed $t>1$, let $\domx=\supp(f)$ be its support and let $\domy=S^{d-1}$. Define
$\psi:\domx\rightarrow \mathbb{R}$ by
$\psi(x)=\log(f(x))$, and define a cost function by
$c(x,y)=-t\log(\max(x\cdot y,0))$.
Then
\begin{enumerate}
    \item We have that $\psi$ is $c$-convex,
    and it admits an explicit transport map
    $T:\domx \rightarrow  \domy$,
   whose image is $\im(T)=\{y\in \domy: \psi^c(y)> -\infty\}$.
\item The inclusion map $\domy(a)\hookrightarrow \domx(a)$ induces 
a homotopy equivalence, where
$\domx(a)=\{\psi \geq a\}$, $\domy(a)=\{\psi^c\geq a\}$
are the super-level sets. The inverse homotopy is represented by the (non-surjective) restriction of $T$ to 
$\domx(a)$.

\item  If $\dataset$ is contained in the sphere $S(V)=V\cap S^{d-1}$ for 
$V\subset \mathbb{R}^d$ a linear subspace, then 
$\psi^c$ is the extension by $-\infty$ of the corresponding conjugate function
$\bar{\psi}^c$ on $S(V)$ associated to 
$\bar{f}=f|_{S(V)}$.
\end{enumerate}
\end{thma}
To prove the first part, we use compatibilities of our construction with respect to restriction to subspaces to reduce to the one-dimensional case of the circle, in which case we also find that the transport map $T$ is increasing as a function of the angle. These compatibilities also easily prove the third statement, which in particular shows that the 
$\epsilon$-covering numbers of $\domy(a)$
are unchanged by taking linear isometric embeddings, unlike those of $\domx(a)$. This is a key property which is special to our choice of cost and kernel function, and in fact this will not hold for other natural choices 
(for instance, taking Gaussian kernels 
on the tangent space).

To prove the second statement about the homotopy equivalence, we determine a deformation retract, but not one determined by $T$, whose restriction to $\domx(a)$ does not surject onto $\domy(a)$. Instead, similar to what was done in \cite{carlsson2024kernel}, we deduce that $T$ induces a homotopy equivalence 
from the preimage 
$T^{-1}(\domy(a))\twoheadrightarrow \domy(a)$, 
which is surjective (and in fact 
is injective when $\supp(f)=S^{d-1}$). 
We then define a deformation retraction
of $T^{-1}(\domy(a))$ onto $\domx(a)$,
making use of the explicit formula
$\psi^c(T(x))=\psi(x)+c(x,T(x))$,
which can be deduced from
the first item and \eqref{eq:introcconv}.

%\section{Main constructions}

\subsection{Application to a Neurological data set}

We apply our construction to a recent 
study involving TDA and the brain activity of rats \cite{Gardner2022}.
Among other things, 
that article found that signals from
certain modules of 100-200 
neurons in the cortex
formed a toroidal shape, which was confirmed by showing that the persistent cohomology formed the Betti numbers of
$(\beta_0,\beta_1,\beta_2)=(1,2,1)$.
Part of that analysis required sophisticated methods to obtain a subsample lying near a smooth manifold, based on a combination of $k$-nearest neighbors density estimation, a method known as topological denoising, and certain components of the UMAP algorithm
\cite{carlsson2009denoising,McInnes2018}.

In this paper, 
we effectively use the transport map 
$T$ and the convex conjugate bound $\psi^c\geq a$ to replace the
manifold approximation and subsampling steps.
%which carries theoretical benefits of item \ref{item:homotopy}. 
The essential hyperparameters in this setup
are the values of
$t,a$ from the theorem, as well as a real number
$0<s<1$ which controls the separation between vertices. In particular, there are no parameters whose optimal value is sensitive to the 
sample size $N$, 
such as the value of
$k$ in the $k$-nearest neighbors step.
By constructing an alpha complex on the resulting vertex set, we find that we 
%are able to 
recover the homology of a torus exactly (without using persistence), using 
a simplicial complex with around 
200 vertices.
We then present some low-dimensional embeddings which reveal a clear 
toroidal shape, calculated from only the 1-skeleton of an alpha complex with a somewhat higher vertex count. This is done using a method based on the Metropolis algorithm, with a loss function based on KL-divergence, similar to $t$-SNE. 
%For this, we used a method based on KL-divergence which is similar in some ways to $t$-SNE, but is based on Monte-Carlo sampling with a temperature parameter to avoid local optima.

\subsection{Acknowledgments}

E. Carlsson was partially supported by (ONR)
N00014-20-S-B001 during this project, which he gratefully acknowledges.

\section{Background on optimal transport}

We recall some definitions from optimal transport theory, referring to \cite{villani2008old} for more details.

\begin{defn}
\label{def:cconv} 
Let $\domx$ and $\domy$ be sets, and let
$c:\domx\times \domy\rightarrow \rplus$
be a cost function. A function
$\psi:\domx\rightarrow \mathbb{R}$ is said to be $c$-convex if 
%it is not identically $\infty$, and 
there exists a function $\zeta:\domx\rightarrow \mathbb{R} \cup \{-\infty\}$ satisfying
\begin{equation}
    \label{eq:cconv}
\psi(x)=\sup_{y\in \domy} (\zeta(y)-c(x,y)).
\end{equation}
In this case, we have the $c$-transform
\begin{equation}
\label{eq:ctransform}
\psi^c(x)=\inf_{x\in \domx} (\psi(x)+c(x,y)),
\end{equation}
which satisfies $\psi^{cc}=\psi$, where
\begin{equation}
\label{eq:psicc}\psi^{cc}(x)=\sup_y \inf_{x'} \left(\psi(x')+c(x',y)-c(x,y)\right).
\end{equation}
The \emph{$c$-subdifferential} is the set
\begin{equation}
\label{eq:subdiff}    
\partial_c \psi=\left\{(x,y) \in \domx\times \domy:
\psi^c(y)-\psi(x)=c(x,y)\right\}.
\end{equation}

\end{defn}
For simplicity, 
we are excluding infinity from the codomain of $\psi$, 
as our functions are finite-valued.

In optimal transport, $c$-convex functions can be used to construct transport maps, whose graph determines the subdifferential
$\partial_c \psi=\{(x,T(x))\}$. 
In the smooth setting,
this is done by translating 
\eqref{eq:subdiff} into the differential criteria
    \begin{equation}
\label{eq:transgrad}    
\grad \psi(x)+\grad_x c(x,y)
\end{equation}
whenever $y=T(x)$. If the function
$\grad_x c(x,\_ )$ is injective for all $x$, a property known as the \emph{twist condition}, then \eqref{eq:transgrad} is enough to uniquely characterize $T$. See the discussion 
at the beginning of Chapter 10 of \cite{villani2008old}.

Transport maps appear as solutions the Monge problem, which seeks to minimize the integral of 
\[\int_{\domx} c(x,T(x)) d\mu(x),\] 
subject to the constraint that $T_*(\mu)=\nu$
for two input measures $\mu,\nu$ on $\domx$ and $\domy$.
In this paper, the reasoning is reversed, as the function $\psi$ is essentially the input. 
The measures $\mu,\nu$ are secondary, but turn out to be important for sampling in a coordinate-independent way.

\section{Alpha shapes for spherical point clouds}

We define kernel density estimators on the sphere, as well as our corresponding transport maps. We then state and prove our main result,
which is Theorem \ref{thm:cconv}.

\subsection{Kernel functions for cosine similarity}

If $V$ is an inner product space, we will write $S(V)$ for the unit sphere in $V$, and
also let $S^{d-1}$ denote the unit sphere in $\mathbb{R}^d$.  
For any $x\in S^{d-1}$,
let $\hemi(x)=\left\{y\in S^{d-1}: x\cdot y\geq 0\right\}$ denote the hemisphere centered at $x$, and let
$\hemistrict(x)$ denote its interior, in which we have strict inequality.
We also have the normalization map $\rho: \mathbb{R}^d-\{0\}\rightarrow \thesphere$ given by
$\rho(x)=x/\lVert x\rVert$.
For any $x \in \thesphere$, we have a map
$\phi_x:\hemistrict(x)\rightarrow T_x S^{d-1}$ 
defined by $\phi_x(y)=y/(y\cdot x)$, identifying $T_x S^{d-1}$ with
its image as an affine subspace in $\mathbb{R}^d$.

Let $\dataset=\{x_i\}_{i=1}^N \subset S^{d-1}$ be a collection of points,
and let $a_i>0$ be a collection of postive weights.
For each scale $t>1$,
we define a kernel-based density estimator on the sphere by powers of the cosine similarity:
\begin{equation}
\label{eq:kerdens}
    f(x)=\sum_{i=1}^N a_i K_t(x,x_i),\ \ 
         K_t(x,y)= \begin{cases}
    (x\cdot y)^{t} & x\cdot y>0\\
    0 & \mbox{otherwise}
\end{cases}
\end{equation}
%where for $x,y \in S^{d-1}$ we define
% \begin{equation}
% \label{eq:cosimker}
%     K_t(x,y)= \begin{cases}
%    (x\cdot y)^{t} & x\cdot y>0\\
%    0 & \mbox{otherwise}
%\end{cases}
% \end{equation}
In connection with the scale parameter in the Gaussian kernel density estimator \eqref{eq:gausskde}, one might write $t=1/h^2$.
The condition that $t>1$ is important because it 
means that $f(x)$ is continuously differentiable.
 
We next define a 
cost function $c=c_t$ by
\begin{equation}
\label{eq:defcost}    
c(x,y)=\begin{cases}
-t\log(x\cdot y) & x\cdot y>0 \\
\infty & \mbox{otherwise}
\end{cases}
\end{equation}
which has also appears in the context of the embedding $S^{d-1}$ into $\mathbb{R}^d$ with given Gaussian curvature \cite{OLIKER2007600}.
Notice that $c(x,y)$ does not satisfy the triangle inequality since we may have $c(x,y)=\infty$,
but $c(x,z),c(y,z)<\infty$.

The cost function has the property that 
the gradient determines a bijection
\begin{equation}
\label{eq:costbij}    
\hemistrict(x)\leftrightarrow T_x S^{d-1},\ \ 
y\mapsto \grad_x c(x,y),
\end{equation}
so that $c(x,y)$ satisfies the twist property.
We have next the formula for the gradient
\begin{equation}
    \grad \psi(x)=tf(x)^{-1}\sum_i K_t(x,x_i)\phi_x(x_i).
\end{equation}
It follows easily that the transport map 
$T:\supp(f)\rightarrow \thesphere$ defined by
\begin{equation}
\label{eq:deftrans}    
T(x)=\rho\left(f(x)^{-1}
\sum_{i} K_t(x,x_i) \phi_x(x_i)\right)
\end{equation}
is the unique one satisfying \eqref{eq:transgrad}.
Notice that the inner expression is a convex combination of
the points $\phi_x(x_i)$.

We can now state our main result:

\begin{thm}
\label{thm:cconv}
Fix $t>1$, and let $f(x)$ and $c(x,y)$
be as in \eqref{eq:kerdens} and \eqref{eq:defcost}.
Define $\psi:\domx\rightarrow \mathbb{R}$ by $\psi(x)=\log(f(x))$, and let $T:\domx\rightarrow \domy$ be as in 
\eqref{eq:deftrans}, where 
$\domx=\supp(f) \subset S^{d-1}$, and
$\domy=S^{d-1}$.
Then we have
\begin{enumerate}
    \item \label{item:cconv} $\psi$ is $c$-convex, and its subdifferential 
    is given by
    %the graph of $T$,%for the cost function \eqref{eq:defcost}, and
\[\partial_\psi=\{(x,T(x)):x\in \domx\}\] 
The image of $T$ the set of $y\in S^{d-1}$ for which $\psi^c(y)\neq -\infty$.
\item \label{item:homotopy} 
For each $a\in \mathbb{R}$, the inclusion map $\domy(a) \hookrightarrow \domx(a)$ 
induces a homotopy equivalence, 
where $\domx(a)=\{\psi\geq a\}$,
$\domy(a)=\{\psi^c \geq a\}$ are the super-level sets. The inverse homotopy is represented by the restriction $T:\domx(a)\rightarrow \domy(a)$.
%(which may not be surjective).
\item \label{item:alphacurse} If $\dataset \subset S(V)$ where $V\subset \mathbb{R}^d$ is a subspace, then 
$\psi^c$ is the extension by $-\infty$ of the map $\bar{\psi}^c:S(V)\rightarrow \mathbb{R} \cup \{-\infty\}$ corresponding to
$\bar{f}=f|_{S(V)}$.
\end{enumerate}
\end{thm}

The $c$-convexity from item \ref{item:cconv} is illustrated in the one-dimensional case in Figure \ref{fig:cconv}. As mentioned in the introduction, the restricted map transport map $T:\domx(a)\rightarrow \domy(a)$ is not surjective in item \ref{item:homotopy}.
The argument we use involves finding a deformation retract of $T^{-1}(\domy(a))\rightarrow \domx(a)$, taking advantage of the explicit formula for 
$\psi^c(T(x))=\psi(x)+c(x,y)$ determined by the first item and \eqref{eq:subdiff}.
While item \ref{item:alphacurse}
may look like a predictable compatibility result,
it implies the crucial property that the covering number of $\domy(a)$ is unaffected by isometrically embedding $\dataset$ into higher dimensions, which is special to the choice of kernel function.
For instance, replacing the kernel function with 
a Gaussian based along the tangent direction would not have this property.

\subsection{Proof of Theorem \ref{thm:cconv}}

We begin with a sequence of lemmas.

\begin{lemma}
\label{lem:rescosim}
Let $V\subset \mathbb{R}^d$ be a subspace, and let $\bar{f}$ be the restriction of $f$ to $S(V)$.
Then
   \begin{enumerate}
    \item \label{item:reskde} We have that
\begin{equation}
\label{eq:reskde}
\bar{f}(x)=\sum_{i} \bar{a}_i 
K_t(x,\bar{x}_i),\ \  \bar{a}_i=\lVert x'_i \rVert^t a_i,\ \ \bar{x}_i=\rho(x'_i).
\end{equation}   
where $x'_i=\proj_V(x_i)$, and the sum is over $i$ for which 
$x'_i\neq 0$.
\item \label{item:restrans} The transport map 
$\bar{T}:\supp(\bar{f}) \rightarrow S(V)$ corresponding to $\bar{f}$
is given by
\begin{equation}
\label{eq:restrans}    
\bar{T}(x)=\rho(\proj_V(T(x)))
\end{equation}
for all $x \in \supp(\bar{f})$.
\item \label{item:subtrans} 
If $\dataset\subset S(V)$,
then we have $\bar{T}=T|_V$, and
\begin{equation}
\label{eq:domtrans}
T(x)=T(\rho(\proj_V(x)))
\end{equation}
for all $x\in \supp(f)$.
\end{enumerate} 
\end{lemma}

\begin{figure}[th]
    \centering
    \includegraphics[width=1\linewidth]{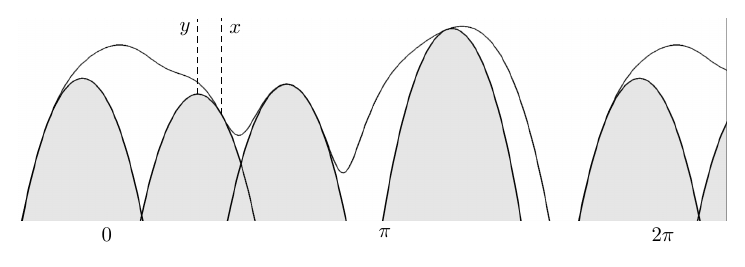}
    \caption{Illustration of the $c$-convexity statement in the one-dimensional case, drawn in angular coordinates.
    The top curve is
    $\psi(x)=\log(f(x))$, where $\dataset$ consists of 10 random points on $S^1$, and $t=1/h^2$ for $h=.3$. 
    The functions which are the upper boundary of the gray regions are given by $\psi^c(y)-c(\_,y)$, whose peak occurs at the value
    $y=T(x)$, where $x$ is the point of contact with $\psi(x)$. The $c$-convexity implies that those functions never exceed the height of $\psi$.}
    \label{fig:cconv}
\end{figure}

\begin{proof}
These are straightforward to deduce from 
\eqref{eq:deftrans}.
\end{proof}

In particular, this allows us to reduce several statements to the one-dimensional case when $f$ is a function on the circle.
Let us denote 
$\gamma(u)=(\cos(u),\sin(u))$,
and identify the hemispheres
with open intervals
$\hemistrict(\gamma(v))=\gamma(I_{v,\pi/2})$
where $I_{v,c}=(v-c,v+c)$.
\begin{lemma}
\label{lem:onedim}    
In the one-dimensional case, 
$f: S^1\rightarrow \mathbb{R}_{+}$, we have
\begin{enumerate}
\item \label{item:halfcirc} The support of $f$ is the union of open half-circles 
$\hemistrict(y)$ as $y$ ranges over
the image of $T$.
\item \label{item:transtilde} 
We have a unique lifting to a continuous 
 $\tilde{T}:\gamma^{-1}(\supp(f))\rightarrow \mathbb{R}$
satisfying $\gamma(\tilde{T}(u))=T(\gamma(u))$, and $|\tilde{T}(u)-u|< \pi/2$ for $u\in I_{v,\pi/2}$.
    \item \label{item:monotone}      
The lifting $\tilde{T}$ is weakly increasing,
and is constant on a neighborhood $I_{v,\epsilon}$
precisely when there is a unique data point
$x_i\in \gamma(I_{v,\pi/2+\epsilon})$,
in which case
its value is the unique preimage
$u_i=\gamma^{-1}(x_i)$ in $I_{u,\pi/2}$.
    \item \label{item:uniquemax} Whenever $y^*=T(x^*)$, the function
$\psi(x)+c(x,y^*)$ achieves its minimum value at $x=x^*$.
\end{enumerate}
\end{lemma}

\begin{proof}
 
For item \ref{item:halfcirc}, let $y=T(x)$
for $x\in \supp(f)$. Then identifying $\hemistrict(x)$ with the interval $I_{u,\pi/2}$, and let $\{u_1,...,u_k\} \in I_{u,\pi/2}$ be the images of $\dataset \cap \hemistrict(y)$ under $\gamma^{-1}$, which must be nonempty. By the form of $T$, we see that $v$ must be in the convex hull of the $u_i$, showing that $\hemistrict(y)\subset \supp(f)$.
On the other hand, every $x\in \hemistrict(y)$, so we have covered the whole support.

The existence of the lifting in item \ref{item:transtilde} is also clear. Explicitly, the angular version of $f(x)$ is given by
\[\tilde{f}(u)=
\sum_i a_i \max(\cos(u-u_i),0)^t\]
from which we obtain
\begin{equation}
\label{eq:transtildeang}    
\tilde{T}(u)=u+
\tan^{-1} \left(\tilde{f}(u)^{-1}\sum_i a_i \max(\cos(u-u_i),0)^t \tan(u_i-u)\right).
\end{equation}

For the last two parts, we can write
the function $\psi(x)+c(x,y)$
in angular coordinates as
\[A(u,v)=\log\left(\sum_{i} a_i \max(\cos(u-u_i),0)^t\right)-t\log(\cos(u-v)),\] 
which is defined on the domain $\{(u,v) \in \supp(\tilde{f})\times \im(\tilde{T}):
u\in I_{v,\pi/2}\}$. We begin by calculating its partial derivatives.

The first partial derivative is given by
\[A_u(u,v)=-t\tilde{f}(u)^{-1} \sum_i a_i \tan(u-u_i) 
\max(\cos(u-u_i),0)^{t}+t\tan(u-v),\]
which is continuous since $t>1$.
It is convenient to write the first term as an expectation
\begin{equation} 
\label{eq:auexpect}
A_u(u,v)=-t\langle \tan(u-u_i) \rangle+t\tan(u-v)
\end{equation}
where $\langle \tan(u-u_i) \rangle$ is the expectation of the random variable $\tan(u-u_i)$
over the normalized finite probability measure
$\tilde{K}_t(u,u_i)/\tilde{f}(u)$ on $\{u_i\}$.

The second partial derivative is given by 
\[A_{uu}(u,v)=t(t-1)\tilde{f}(u)^{-1}\sum_{i} a_i \tan^2(u-u_i)
\max(\cos(u-u_i),0)^{t}-\]
\[t^2(\tilde{f}(u)^{-1}\sum_i a_i \tan(u-u_i) 
\max(\cos(u-u_i),0)^t)^2+\]
\[-t+t(1+\tan^2(u-v)),\]
which is valid when $u$ is in the complement of $\{u_i\pm \pi/2+2k\pi\}$, since the first term has a singularity at these points if $t\leq 2$.
Rewriting this expression using expectations as above, we obtain
\[t(t-1)\langle \tan(u-u_i)^2\rangle-t^2\langle \tan(u-u_i)\rangle^2+t\tan^2(u-v)=\]
\[t(t-1) \left( \langle \tan(u-u_i)^2\rangle-
\langle \tan(u-u_i)\rangle^2\right)+\]
%A_u(u,v)\]
\[A_u(u,v) (\langle \tan(u-u_i)\rangle+\tan(u-v)).\]
%by inserting \eqref{eq:auexpect}.
%t(\tan^2(u-v)-\langle \tan(u-u_i)\rangle^2).\]
using \eqref{eq:auexpect}.
We then find that $A_{uu}(u,v)\geq 0$ when 
$v=\tilde{T}(u)$, since the first expression is
non-negative for $t>1$ by Cauchy-Schwarz, and $A_u(u,\tilde{T}(u))=0$.

To check that $\tilde{T}$ is increasing, we have
\[\frac{\partial}{\partial u} A_u(u,\tilde{T}(u))=
A_{uu}(u,\tilde{T}(u))+A_{uv}(u,\tilde{T}(u))\tilde{T}'(u).\]
We have already established that $A_u(u,\tilde{T}(u))$
is identically zero, and
by the above paragraph, we have that 
$A_{uu}(u,\tilde{T}(u))\geq 0$ where it is defined. Next, we have that
$A_{uv}(u,v)=-t\sec^2(u-v)$ which is negative
for $v=\tilde{T}(u)\in (u-\pi/2,u+\pi/2)$, so that 
$\tilde{T}'(u)$ is defined and nonnegative whenever $u$ is in the complement of
$\{u_i\pm \pi/2+2k\pi\}$. Since 
the complement is dense and
$\tilde{T}$ is continuous, it must be 
weakly increasing on its domain. The second part of item \ref{item:monotone} follows since we have positivity in Cauchy-Schwarz when there is more than one distinct point in range.

For the final statement, suppose that
$y^*=T(x^*)$, and that 
\[\psi(x)+c(x,y^*) <\psi(x^*)+c(x^*,y^*)\]
for some $x \in \hemistrict(y^*)$.
We may assume that $x$ is a
global minimizer, so that 
$\grad \psi(x)+\grad_x c(x,y^*)=0$, and therefore $T(x)=y^*$.
%recalling that $\hemistrict(y^*)\subset \supp(f)$. 
Since $x^*,x\in \hemistrict(y^*)$,
we may take $I_{v^*,\pi/2}$ to be the corresponding interval with $y^*=\gamma(v^*)$, and let $u,u^*$ be the points in the interval corresponding to $x,x^*$. Then $\tilde{T}$ is defined on 
the whole interval by the item \ref{item:halfcirc},
it is increasing by item \ref{item:monotone}, and we have $T(u)=T(u^*)=v^*$.
It must therefore be constant and equal to $v^*$ between $u$ and $u^*$, from which it follows that $A(\_ ,v^*)$ is constant on that interval by the second part of item \ref{item:monotone}, a contradiction.

\end{proof}

\begin{remark}
    The statements that $A_u(u,T(u))=0$ and $A_{uu}(u,T(u))\geq 0$ from the proof below also appear in the differential criteria for $c$-convexity, which is the same Theorem 12.46 of \cite{villani2008old}. The reason we cannot apply that theorem directly is that it assumes that $\psi$ is $C^2$, which would require that $t > 2$.
\end{remark}

\begin{lemma}
\label{lem:cconv}
We have that $\psi$ is $c$-convex, where $\psi(x)$ and $c(x,y)$ are as in Theorem \ref{thm:cconv}.
\end{lemma}

\begin{proof}

We first show that whenever $y^*=T(x^*)$, we have that 
\begin{equation}
\label{eq:cconvineq}
    \psi(x^*)+c(x^*,y^*)\leq 
    %=x^*$ minimizes
\psi(x)+c(x,y^*).
\end{equation}
%(see the remark after the statement of Lemma \ref{lem:onedim}). 
for all $x\in \hemistrict(y^*)$.
Suppose then that we have the reverse inequality in \eqref{eq:cconvineq} for some
$x\neq x^*$,
and let $V$ be the span of $x$ and $x^*$.
Then the inequality is also satisfied with $\rho(\proj_V(y^*))$ in place of $y^*$, since the substitution subtracts the same constant from both sides by the form of $c$. Since this projection is the same as $\bar{T}(x^*)$ by item \ref{item:restrans} of Lemma \ref{lem:rescosim},
we have reduced \eqref{eq:cconvineq} to 
the one-dimensional case, which follows from
item \ref{item:uniquemax} of Lemma \ref{lem:onedim}.

The $c$-convexity now follows from standard arguments related to the Monge problem in the differentiable case
(see the beginning of Chapter 10 of \cite{villani2008old}, as well as the proof of Theorem 12.46): 
first, it follows from the previous paragraph that $\{(x,T(x))\}\subset \partial_c\psi$. To check the reverse inclusion, suppose $(x,y)\in \partial_c\psi$. Since $\psi$ is in
$C^1$ for $t>1$, 
it can be deduced from taking the limit 
as $x\rightarrow x^*$ from different tangent directions 
in \eqref{eq:cconvineq} that 
$\grad \psi(x)+\grad_x c(x,y)=0$. Then 
since the map $y\mapsto \grad_x c(x,y)$ is a bijection $
\hemistrict(x) \leftrightarrow T_x S^{d-1}$, we have that
$y=T(x)$ is the unique value for which $(x,y)\in\partial_c\psi$.
To check the $c$-convexity, if the maximal value takes place at $x=x^*$, this just means that 
$\psi^c(y^*)=\psi(x^*)+c(x^*,y^*)$, from which it follows that $\psi^{cc}=\psi$, where
$\psi^{cc}$ is as in \eqref{eq:psicc}.
    
\end{proof}

We now have the following useful formula, which follows from the Lemma \ref{lem:cconv} and the definition of $\partial_c \psi$.
\begin{equation}
\label{eq:psicoft}    
\psi^c(T(x))=
\psi(x) + c(x,T(x)).
\end{equation}

\begin{lemma}
\label{lem:contractible}

The inclusion map $\iota:\domy(a)\hookrightarrow T^{-1}(\domy(a))$ induces a homotopy equivalence whose inverse is represented by the restriction of $T$.

\end{lemma}

\begin{proof}

We show that the level sets $T^{-1}(y)$ for $y\in \thesphere$ are either empty or contractible.  
First, we have that $T^{-1}(y)$ is entirely contained in $\hemistrict(y)$.
It suffices to show that for
any points satisfying $T(x)=T(x')=y$,
the unique geodesic arc in $\hemistrict(y)$
connecting them is entirely 
contained in the level set, for then the region
is geodesically convex. This follows from Lemma
\ref{lem:rescosim} applied to the subspace $V$
spanned by $x,x'$ together with item \ref{item:monotone} of Lemma \ref{lem:onedim}.

To show that the inverse equivalence is induced by the inclusion map, we define a homotopy $H:T^{-1}(\domy(a))\times [0,1]\rightarrow T^{-1}(\domy(a))$
between $\iota T$ and the identity, given by
\begin{equation}
\label{eq:defhomotopy}    
H(x,s)= \rho((1-s)x+sT(x)).
\end{equation}
What has to be checked is that
$H(x,s)$ remains in $T^{-1}(\domy(a))$
for all $s$.

Let $y=T(x)$ for some 
$x \in T^{-1}(\domy(a))$.
Using \eqref{eq:psicoft}, 
it suffices to show that for any $z$ 
on the arc connecting $x$ and $y$, we have
\begin{equation}   
\label{eq:cconvarc}
\psi(x)+c(x,T(x))\leq \psi(z)+c(z,T(z))
\end{equation}
First, we easily have that
\begin{equation}
\label{eq:homotopytoonedim} 
\psi(z)+c(z,\rho(\proj_V(T(z))))\leq \psi(z)+c(z,T(z)),
\end{equation}
where $V$ is the two-dimensional subspace spanned by $x$ and $y$, which 
also contains $z$. Then by item \ref{item:restrans} of Lemma \ref{lem:rescosim}, we have reduced the problem to checking \eqref{eq:cconvarc} in the one-dimensional case.

In the one-dimensional case, we have that
\begin{equation}
\label{eq:homotopyonedim}   
\psi(z)+c(z,T(x))\leq 
\psi(z)+c(z,T(z))
\end{equation}
by the increasing property from item \ref{item:monotone} of Lemma \ref{lem:onedim},
which guarantees that $T(z)$ is at least as far from $z$ as $y$, since $z$ 
lies between $x$ and $y$.
Equation \eqref{eq:cconvarc} follows 
since $x$ is a global minimizer of
$\psi(\_)+c(\_,y)$ by Lemma \ref{lem:cconv}.

\end{proof}

We can now prove Theorem \ref{thm:cconv}.

\begin{proof}

Item \ref{item:cconv} has been proved in 
Lemma \ref{lem:cconv}.

Next, it follows from 
Lemma \ref{lem:contractible} that
$\domy(a) \sim T^{-1}(\domy(a))$.
Then to prove the homotopy equivalence statement in item \ref{item:homotopy}, 
it suffices to produce a deformation retraction
\[\ret : T^{-1}(\domy(a)) \times [0,1]
\rightarrow T^{-1}(\domy(a))\]
of $T^{-1}(\domy(a))$ onto $\domx(a)$, which is defined as follows:
for each $x\in T^{-1}(\domy(a))$, let
$x'$ be the closest point to $x$ on the ball
$B_c(y; \psi^c(y)-a)$, where $y=T(x)$ and 
$B_c(y,a)=\left\{c(\_, y)\leq a\right\}$.
The ball is nonempty precisely 
%at those points in the domain.
when $x\in T^{-1}(\domy(a))$.
Then we define
\begin{equation}
\label{eq:defret}    
F(x,s)= \rho((1-s)x+sx').
\end{equation}
See Figure \ref{fig:defret}
for an illustration in the two-dimensional case.

By the proof of Lemma \ref{lem:contractible}, every point on the arc connecting $x$ and $y$
(and therefore the arc connecting $x$ and $x'$)
remains in $T^{-1}(\domy(a))$, so 
that $F(x,s) \in T^{-1}(\domy(a))$ for all $s$.
By definition, we have that $F(x,0)=x$ for all $x$.
Since $\psi^c(y)-c(\_,y)\leq \psi(\_)$ and $x'$ is on the ball, we find that 
$\psi(x')\geq a$, so that $F(x,1)\in \domx(a)$.
It follows from \eqref{eq:psicoft} that $x\in B_c(T(x);\psi^c(x)-a)$ whenever $x\in \domx(a)$, so that
$F(x,s)=x$ on that range.

%$u'' \leq u' \leq \tilde{T}(u) \leq \tilde{T}(u'')$,
%using the corresponding variables in angular coordinates.

The last item is a consequence of Lemma \ref{lem:rescosim}.

\end{proof}

\begin{figure}[t]
    \centering    
\begin{subfigure}[b]{.49\textwidth}
 \centering
\includegraphics[scale=.8]{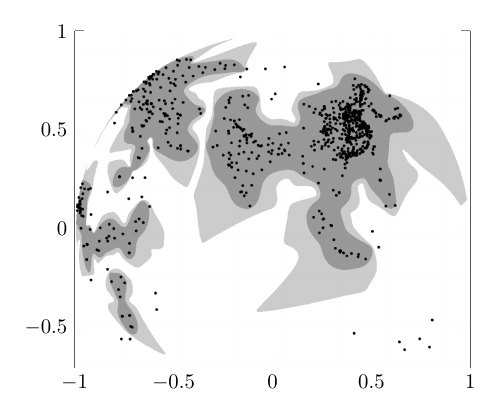}
\end{subfigure}
    \begin{subfigure}[b]{.49\textwidth}
    \centering
\includegraphics[scale=.8]{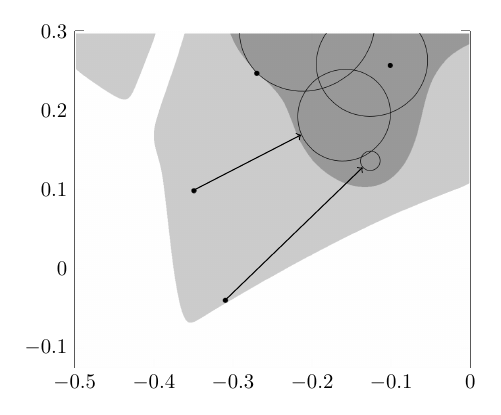}
\end{subfigure}
    \caption{Illustration of the deformation retract $F$ from the proof of Theorem \ref{thm:cconv}, where $\dataset\subset S^2$ 
    is a point cloud of the 1000 most populated world cities \cite{WCD}. 
    The left frame shows the point cloud, as well as the regions $T^{-1}(\domy(a))$ (light gray),
    and $\domx(a)$ (dark gray).
    The retract, shown on the right, moves along the arrow connecting each point $x$ to its nearest point in the ball
    $\left\{c(\_ ,y)\leq \psi^c(y)-a\right\}\subset \domx(a)$ for $y=T(x)$. The ball is nonempty precisely when
   $x\in T^{-1}(\domy(a))$, and the radius shrinks to zero as $x$ approaches the boundary. The trajectory is constant when we start with $x\in \domx(a)$, in which case $x$ is already in the ball. }
    \label{fig:defret}
\end{figure}

\section{Application to a neurological data set}

We summarize a straightforward algorithm for generating finite alpha complexes from 
$f(x)$, which we propose as a tractible but rigorous method for noise reduction and subsampling.
We apply this method to a point cloud 
generated from a remarkable recent article in neuroscience, which finds a toroidal structure in the brain activity of certain  grid cells in rats \cite{Gardner2022}. Software for reproducing this application is described in Section \ref{sec:software} below.

\subsection{Sampling algorithm}

We explain an algorithm for sampling landmarks point from the super-level sets $\domy(a)$, which serve as the vertices of an alpha complex.

Let $\mu$ be the underlying distribution whose density is $f(x)$, with respect to the uniform measure on the sphere. In order to sample from $\mu$, suppose we have chosen $x\in S^{d-1}$. The distribution resulting from sampling 
$\tilde{x}\in \hemistrict(x)$ 
uniformly at random and recording
$c=\tilde{x}\cdot x$ is given by the density
$(1-c)^{(d-1)/2-1}$ with respect to the uniform measure on $[0,1]$. Then scaling this by the kernel function 
$(\tilde{x}\cdot x)^t$ results in a Beta distribution.
We may then sample from $\mu$ by selecting a random index $i$ in proportion to the coefficient $a_i$, 
and then sampling from $\hemistrict(x_i)$ by picking a random great circle through $x_i$, 
and taking $\tilde{x}$ to be the point on that circle whose dot product is given by a sample $\tilde{c}$ as above.

The next proposition shows that this pushforward measure under the transport map coordinate invariant in the same sense as $\psi^c$.
\begin{prop}
Let $\nu=T_*(\mu)$ be the pushforward measure,
and suppose $\dataset\subset S(V)$ for a subspace $V$. 
Then $\nu$ is supported on $S(V)$, and is equal to 
    $\bar{\nu}=\bar{T}_*(\bar{\mu})$ associated to 
    the restriction $\bar{f}$ as in Lemma \ref{lem:rescosim}.
\end{prop}

\begin{proof}

Repeating the statement, 
we may assume that $V$ has codimension 1.
For each $x_i\in \dataset$, let us parametrize
$\hemistrict(x_i)$ by a point 
$\bar{x}\in \bar{H}^+(x_i)=\hemistrict(x_i) \cap S(V)$ together with an inclination
$\theta \in (0,\pi)$. Then the resulting uniform measure on $\hemistrict(x_i)$ can be described by
$\sin(\theta)d\theta$ times the uniform measure on
$\bar{H}^+(x_i)$. The kernel also factors as
$K_t(x,x_i)=K_t(\bar{x},x_i) \sin^t(\theta)$, so we see that the underlying distribution $\bar{\mu}$ of 
$\bar{f}(x)$ is the pushforward of $\mu$ under the projection map $x\mapsto \rho(\proj_V(x))$. Then the statement follows from 
item \ref{item:subtrans} of Lemma \ref{lem:rescosim}.

\end{proof}

Given $f(x)$, we now have an algorithm for generating an alpha complex
which models $\domy(a)$, similar to the one from \cite{carlsson2024kernel}.
Recall that the unweighted alpha complex 
$\Alpha(S,r)$ is defined as the nerve of the covering by balls of radius $r$ centered at the vertices of $S$, intersected with the corresponding Voronoi cells at those same points \cite{edelsbrunner2010computational}.
\begin{alg} Suppose we are given 
$\dataset,t,a$ as in Theorem \ref{thm:cconv}, as well as a separation parameter $0<s<1$, and a sampling number $M\gg 0$. We generate an alpha complex as follows:
%Generate an alpha complex.
\label{alg:sampshape}
    \begin{enumerate}
    %\item Choose a real number $a$, and fix a minimal separation parameter $0<s<1$.
    \item Sample $M$ points $\{\tilde{x}_1,...,\tilde{x}_M\}$ from the distribution of $\mu$ as above,
    and simulataneously compute the value
    $\tilde{y}_i=T(\tilde{x}_i)$, as well as
    $\tilde{a}_i=\phi^c(\tilde{y}_i)$ using 
    \eqref{eq:psicoft}.
%    $\tilde{x}_i$ from the distribution of $f$, and
    %Sort them in decreasing order of $a_i=\psi^c($
\item  Initialize a vertex set $S=\emptyset$.
for each $\tilde{y}_i$ in decreasing order of 
$\tilde{a}_i$, add $\tilde{y}_i$ to $S$ if
$\tilde{a}_i\geq a$ and 
$K_t(\tilde{y}_i,\tilde{y_j})\geq s$ for all $\tilde{y}_j$ already added to $S$.
\item 
Construct the unweighted alpha complex 
$\Alpha(S,r)$ 
associated to $r=\sqrt{2-2s^{-t}}$, the corresponding distance in the sphere,
using the algorithm of \cite{carlsson2023alpha} or another effective algorithm.
\end{enumerate}
\end{alg}
Note that this algorithm would be impractical using $\domx(a),\psi$ instead of $\domy(a),\psi^c$, due the dependence of the $r$-covering number of $\domx(a)$ 
on the embedding dimension $d$.
If $M$ is still not sufficiently large to 
fill out $\domy(a)$, one might also use an enlarged radius of $(1+\epsilon)r$ in the alpha complex.

\subsection{Grid cells}

We illustrate this method using a 
data set from \cite{Gardner2022}, 
which studied neural activity in rats as they moved in a confined space. 
Among other discoveries,
they found that the brain activity of certain modules of \emph{grid cells}, which play a role in measuring physical location, were concentrated along the surface of a torus associated to different lattice structures in $2$-dimensional space. This was confirmed using persistent cohomology, specifically the Vietoris-Rips construction, which found persistent Betti numbers
consistent with a torus, namely $(\betti_0,\betti_1,\betti_2)\sim (1,2,1)$.

In the experimental setup, while the rat was alert and free to move about an open area, the excitations of grid cells, referred to as spikes, were recorded.
Within each of several modules, the spikes were collected into a $\timesteps\times \neurons$ matrix $A$ as a sequence of voltage thresholds from 0 through 4, where $\timesteps$ represents the number of time steps, and $\neurons$ represents the number of neurons. Typical values would involve a few hundred thousand time steps, and about $100-200$ neurons.

Before taking persistent cohomology, several steps were required to process the data. In order, these were: temporal smoothing, which involved convolving the time direction by a Gaussian; subsampling down to $\timesteps=25000$ 
by choosing well-spaced ``active times,''
which were rows of high $L^1$-norm;
normalizing the mean and variance of each of the $\neurons$-columns; dimension reduction using a PCA to dimension 6; additional down-sampling to 
about 1200 points which lie near a low-dimensional manifold, using a combination of $k$-nearest neighbors density estimation, a method known as topological denoising 
\cite{carlsson2009denoising}, and certain aspects of the UMAP algorithm \cite{McInnes2018}. Persistent cohomology was then calculated using the Ripser software
\cite{bauer2021ripser}, revealing the Betti numbers of a 2-dimensional torus.

\subsection{Spherical alpha shapes method}

We used Algorithm \ref{alg:sampshape}
as a replacement for the steps 
related to subsampling and noise reduction,
leaving only elementary preprocessing steps.
Specifically, starting with the above point cloud $A$, for a particularly module with $\timesteps=126728$ and $\neurons=111$, we applied the following:
\begin{enumerate}
    \item Apply Gaussian smoothing in the time direction, using the same code used in \cite{Gardner2022}.
\item Normalize each of the 111 columns to have the mean zero and variance one.
\item Apply an SVD to reduce the dimension to 6.
\item Normalize each of the 126728 rows to have radius 1.
\end{enumerate}
This resulted in a point cloud $\dataset \subset S^5$ of size 126728 in the $5$-sphere. 

We then selected a value of $t=1/h^2$ with $h=.3$, and a cutoff value of $a$ with the property that $\psi(x_i)\geq a$ for about $90\%$ of the points $x_i\in \dataset$.
We then applied Algorithm \ref{alg:sampshape} using the parameter $s=.4$ and $M=100000$ (more than enough) samples
to generate a vertex set $S\subset S^5$ of size only $218$. Raising the value of $M$ did not result in significantly more points. The resulting alpha complex $X=\Alpha(S,r)$, computed up to the 3-simplices using $1.1\times r$ for the radius, had $(\#X_0,\#X_1,\#X_2,\#X_3)=(218,1769,3912,3566)$ simplices in each dimension. We calculated Betti numbers of
$(1,2,1)$ exactly, without persistence.

Next, we took the times at which the rat's brain signal was in the Voronoi cell of a given vertex,
and plotted its known spatial coordinates at that location, revealing a periodic pattern similar to the one predicted in \cite{Gardner2022}. Finally, we raised $s$ to about .6, thereby increasing the number of vertices to about 820. Using the 1-skeleton of the resulting alpha complex
$X'$, we generated a low-dimensional embedding into $\mathbb{R}^3$. This was done using a method based on KL-divergence, with some similarities to $t$-SNE, but using the Metropolis algorithm with a high enough temperature parameter to avoid local optima. 
The results are shown in Figure \ref{fig:donuts}.

\begin{figure}[t]
    \centering    
\begin{subfigure}[b]{.32\textwidth}
    \centering
\includegraphics[scale=.65]{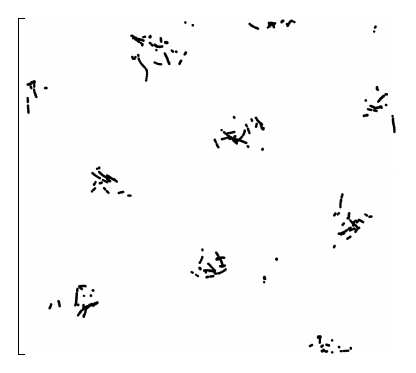}
\end{subfigure}
\begin{subfigure}[b]{.32\textwidth}
 \centering
\includegraphics[scale=.65]{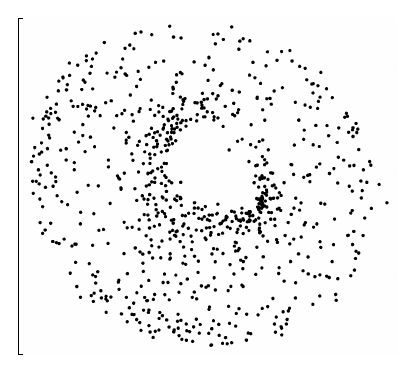}
\end{subfigure}
    \begin{subfigure}[b]{.32\textwidth}
    \centering
\includegraphics[scale=.65]{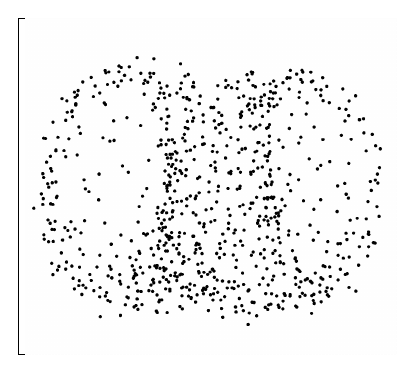}
\end{subfigure}
    \caption
    {In the left frame, the true locations of the rat when the brain signal is in the Voronoi cell of a particular vertex of $X$, an alpha complex with 218 vertices built from only the neural activity. 
    In the middle and right frames, we have two different views of a 3D-embedding generated from the graph associated to the 1-skeleton of $X'$, a similar alpha complex with 820 vertices. }
    \label{fig:donuts}
\end{figure}

\subsection{Further extensions}

We have shown how to apply the constructions resulting from Theorem \ref{thm:cconv} in place of a standard pipeline related to noise reduction and subsampling in TDA.
Using these methods 
provided the benefit of a 
theoretical connection to
density super-level sets, and also
required only 
statistically intrinsic hyperparameters which are stable with respect to changes in the size of the point cloud,
unlike other choices such as the 
``$k$'' in $k$-nearest neighbors.
We also found that the resulting alpha complexes required considerably fewer vertices, and were well-suited for generating low-dimensional embeddings.

However, there may be context-specific changes that could help extend these methods to other modules or other cell types. For instance, one challenge is that the density is not evenly distributed over the torus, so that the cut-off parameter $a$ may have to vary accross other modules. One modification to address this would be to adjust the weights $a_i$ defining $f$ so that that the density of the true positions is uniformly distributed over the enclosed domain.
We hope this explore this and other extensions  in future papers.

\subsection{Software availability}

\label{sec:software}

The implementation of the code for generating the transport maps and spherical 
alpha shapes can be 
found at the first author's website
\url{https://www.math.ucdavis.edu/~ecarlsson/}
under ``Software.'' 
To acquire the data set to implement the grid cells application,
we streamlined the unarchival script originally provided by \cite{Gardner2022} in the corresponding github repository.
Our code is located at 
\url{https://github.com/gdepaul/DensiTDA/tree/main}. In order to run the github, follow the installation steps listed in the 
\texttt{README} file. This will also require downloading the original dataset from
\cite{gardner2022figshre}
%\url{https://figshare.com/articles/dataset/Toroidal_topology_of_population_activity_in_grid_cells/16764508?file=35078602} 
and placing it in the same directory.  

\bibliographystyle{plain}
\bibliography{refs}

\begin{thebibliography}{10}

\bibitem{bauer2021ripser}
Ulrich Bauer.
\newblock Ripser: efficient computation of {V}ietoris-{R}ips persistence barcodes.
\newblock {\em J. Appl. Comput. Topol.}, 5(3):391--423, 2021.

\bibitem{bubenik2012statistical}
Peter Bubenik.
\newblock Statistical topological data analysis using persistence landscapes.
\newblock {\em J. Mach. Learn. Res.}, 16:77--102, 2012.

\bibitem{carlsson2024kernel}
E.~Carlsson and J.~Carlsson.
\newblock Alpha shapes in kernel density estimation, 2024.

\bibitem{carlsson2023alpha}
Erik Carlsson and John Carlsson.
\newblock Computing the alpha complex using dual active set quadratic programming.
\newblock {\em Scientific Reports}, 19824, 2024.

\bibitem{chazal2011clustering}
Frédéric Chazal, Leonidas Guibas, Steve Oudot, and Primoz Skraba.
\newblock Persistence-based clustering in riemannian manifolds.
\newblock {\em Journal of the ACM}, 60, 06 2011.

\bibitem{edelsbrunner1995union}
H.~Edelsbrunner.
\newblock The union of balls and its dual shape.
\newblock {\em Discrete and Computational Geometry}, pages 415--440, 1995.

\bibitem{edelsbrunner1983shape}
H.~Edelsbrunner, D.~Kirkpatrick, and R.~Seidel.
\newblock On the shape of a set of points in the plane.
\newblock {\em IEEE Transactions on Information Theory}, 29(4):551--559, 1983.

\bibitem{edelsbrunner1992weighted}
Herbert Edelsbrunner.
\newblock Weighted alpha shapes.
\newblock In {\em Rept. UIUCDCS-R-92-1760, Dept. Comput. Sci., Univ. Illinois at Urbana-Champaign, Illinois}, 1992.

\bibitem{edelsbrunner2010computational}
Herbert Edelsbrunner and John Harer.
\newblock {\em Computational Topology - an Introduction.}
\newblock American Mathematical Society, 2010.

\bibitem{gardner2022figshre}
Richard Gardner and Erik Hermansen.
\newblock Toroidal topology of population activity in grid cells.
\newblock \url{https://doi.org/10.6084/m9.figshare.16764508.v6}, 2021.
\newblock figshare. Dataset.

\bibitem{Gardner2022}
Richard~J. Gardner, Erik Hermansen, Marius Pachitariu, Yoram Burak, Nils~A. Baas, Benjamin~A. Dunn, May-Britt Moser, and Edvard~I. Moser.
\newblock Toroidal topology of population activity in grid cells.
\newblock {\em Nature}, 602(7895):123--128, Feb 2022.

\bibitem{carlsson2009denoising}
Jennifer Kloke and Gunnar Carlsson.
\newblock Topological de-noising: Strengthening the topological signal.
\newblock {\em ArXiv Preprint ArXiv:0910.5947}, 10 2009.

\bibitem{McInnes2018}
Leland McInnes, John Healy, Nathaniel Saul, and Lukas Großberger.
\newblock Umap: Uniform manifold approximation and projection.
\newblock {\em Journal of Open Source Software}, 3(29):861, 2018.

\bibitem{OLIKER2007600}
Vladimir Oliker.
\newblock Embedding sn into rn+1 with given integral gauss curvature and optimal mass transport on sn.
\newblock {\em Advances in Mathematics}, 213(2):600--620, 2007.

\bibitem{phillips2015geometric}
Jeff~M. Phillips, Bei Wang, and Yan Zheng.
\newblock {Geometric Inference on Kernel Density Estimates}.
\newblock In Lars Arge and J\'{a}nos Pach, editors, {\em 31st International Symposium on Computational Geometry (SoCG 2015)}, volume~34 of {\em Leibniz International Proceedings in Informatics (LIPIcs)}, pages 857--871, Dagstuhl, Germany, 2015. Schloss Dagstuhl -- Leibniz-Zentrum f{\"u}r Informatik.

\bibitem{villani2008old}
Cédric Villani.
\newblock {\em Optimal transport -- Old and new}, volume 338, pages xxii+973.
\newblock Springer Berlin, Heidelberg, 01 2008.

\bibitem{WCD}
WCD.
\newblock World cities database.
\newblock \url{https://simplemaps.com/data/world-cities}.
\newblock Accessed: 2023-12-14.

\end{thebibliography}

\end{document}